\documentclass[a4paper,12pt]{amsart}
\usepackage[utf8x]{inputenc}
\usepackage{amssymb}
\usepackage{amsmath}
\usepackage{mathrsfs}%Serve a poter scrivere in corsivo inglese
\usepackage{bm}
\usepackage{amsthm}%Serve per avere ambienti non numerati
\usepackage{enumitem}%Serve per gli elenchi numerati
\usepackage{color}%Serve per scrivere con i colori
\usepackage{graphicx}
\usepackage{bm}%serve a scrivere in grassetto corsivo
\usepackage{ulem}%serve a barrare o sottolineare
\usepackage{upgreek}

\newtheorem{proposizione}{Proposition}[section]
\newtheorem{teorema}[proposizione]{Theorem}
\newtheorem{lemma}[proposizione]{Lemma}

\renewcommand{\phi}{\varphi}
\newcommand{\de}{\partial}
\newcommand{\R}{\mathbf{R}}
\newcommand{\C}{\mathbf{C}}
\newcommand{\ci}[1]{\mathscr{#1}}%corsivo inglese

\newcommand{\N}[1]{\left\lVert#1\right\rVert}
\newcommand{\bra}{\left\langle}
\newcommand{\ket}{\right\rangle}
\renewcommand{\H}{\mathbf{H}}
\newcommand{\g}[1]{\mathfrak{#1}}
\newcommand{\e}{\varepsilon}

\renewcommand{\l}{\lambda }
\newcommand{\var}{\varphi }

\renewcommand{\th }{\theta }

\DeclareMathOperator{\Vol}{Vol}

\title{A non-convergence phenomenon for the CR Yamabe flow}
\author{Claudio Afeltra}
\address{Université de Montpellier,  
	Institut Montpelliérain  Alexander Grothendieck.
	Place Eug\`ene Bataillon, 
	34090 Montpellier,  France}

\email{claudio.afeltra@umontpellier.fr}

\author{Pak Tung Ho}
\address{Department of Mathematics, Tamkang University Tamsui, New Taipei City 251301, Taiwan}

\email{paktungho@yahoo.com.hk}

\author{Andrea Pinamonti}
\address{Dipartimento di Matematica, Università di Trento\\
Via Sommarive, 14, 38123 Povo TN, Italy}
\email{andrea.pinamonti@unitn.it}


\begin{document}

\begin{abstract}
We construct a contact form on a three dimensional CR manifold such that the CR Yamabe flow fails to converge. More precisely, on small Rossi deformations of the standard
CR three-sphere, we exhibit an example whose corresponding CR Yamabe flow develops a one-bubble concentration regime.

The construction is based on the negativity of the pseudohermitian
mass on the Rossi spheres. This shows that mass positivity is not merely
a technical assumption in the known convergence results for the CR
Yamabe flow, but is genuinely connected to the large-time dynamics of the
flow.
\end{abstract}
\maketitle

\section{Introduction}
Given a compact Riemannian manifold $(M,g)$, the Yamabe problem consists of proving the existence of a metric conformal to $g$ with constant scalar curvature.
This problem, which has played a major role in Geometric Analysis, has been solved through the works of Yamabe, Trudinger, Aubin (who settled the case of non-conformally flat manifolds of dimension $n\ge 6$) and Schoen (who settled the remaining cases).
%\rosso{(PUT CITATIONS?)}
The proof is based on the study of the Yamabe functional
$$\ci{R}(g) = \frac{\int_MR_gdV_g}{\Vol_g(M)^{\frac{n-2}{n}}},$$
whose critical points, when restricted to a conformal class, are solutions of the Yamabe problem.
The final step by Schoen requires the Positive Mass Theorem.% in dimensions $3\le n\le 5$ and for locally conformally flat manifolds

Over the last few decades, flowing methods have attracted a considerable attention in Geometric Analysis, and have been employed to try to solve a lot of problems. The Yamabe problem is no exception, and because of its variational form, it is quite natural to define a flow based on it: indeed, up to an inessential renormalization, the $L^2$  gradient flow of the functional $\ci{R}$ is given by the following equation
$$\frac{\de g(t)}{\de t} = -(R_{g(t)}-r_{g(t)})g(t)$$
where $r_{g(t)}$ is the average scalar curvature of $g(t)$.
It is easy to see that if the Yamabe flow has a solution for $t\in(0,\infty)$ and if $g(t)$ converges to some metric $g_{\infty}$, then $g_{\infty}$ is a solution of the Yamabe problem; therefore it is natural to study whether global existence and convergence hold.

Both questions have been answered positively by Brendle: convergence in low dimensions and long time existence in general in \cite{B-LowDim}, and convergence in every dimension in \cite{B-Convergence} (under the assumption that the Positive Mass Theorem holds in any dimension). %\rosso{(CITE PREVIOUS PARTIAL RESULTS?)}

\vspace{5mm}

These questions arise also naturally in the context of CR manifolds.
CR manifolds are defined as $2n+1$-dimensional manifolds $M$ endowed with a complex $n$-dimensional subbundle $\ci{H}$ of $TM\otimes\C$ such that $[\Gamma(\ci{H}),\Gamma(\ci{H})]\subseteq\Gamma(\ci{H})$ and $\ci{H}\cap\overline{\ci{H}}=\{0\}$. CR manifolds provide an abstract model for real hypersurfaces in complex manifolds.
They are called nondegenerate if $H(M)=\g{Re}(\ci{H}\oplus\overline{\ci{H}})$ is a contact distribution; in such a case the choice of a contact form determines a rich geometric structure on $M$,
%In particular it induces a nondegenerate symmetric bilinear form $G_{\theta}$ on $H(M)$; the CR manifold is called pseudoconvex when $G_{\theta}$ is positive definite (this does not depend by the contact form). It also determines a connection on $M$, called Tanaka-Webster connection.
in particular a metric on $H(M)$ and a connection. By contracting twice the curvature tensor of the connection with the metric, analogously to Riemannian geometry, one gets a scalar curvature invariant, called Webster curvature.
Since a contact form is determined up to multiplication by a nowhere-vanishing function, it is natural to study the problem of finding a contact form such that the Webster curvature is constant, known as CR Yamabe problem.

The CR Yamabe problem shows strong similarities with the Riemannian one; in particular it is equivalent to finding critical points of the CR Yamabe functional
$$Q(\theta)=\frac{\int_MR_{\theta}dV_{\theta}}{\Vol_{\theta}(M)}$$
(see Section \ref{Section2} for the notation).
In  \cite{JL}, Jerison and Lee solved the problem for CR manifolds of dimension $2n+1\ge 5$ not locally equivalent to $S^{2n+1}$ using a strategy analogous to Aubin’s in the Riemannian case, while Gamara and Yacoub in \cite{G,GY} settled the remaining cases; since no analogue of the ADM mass was available at the time, they used Bahri and Coron's method of critical points at infinity, a topological method based on the deformation of sublevels of the functional $Q$ by adding virtual critical points in order to bypass the failure of the Palais-Smale condition. In particular they proved the existence of critical points of $Q$, but not of minimizers.

In \cite{CMY1} Cheng, Malchiodi and Yang defined a notion of mass for asymptotically flat CR manifolds (in dimension three, but easily generalizable to higher dimensions as was done in subsequent works), and proved a Positive Mass Theorem
% for three-dimensional CR manifolds
%arising as the blow-up of a compact manifold of positive CR Yamabe class (which is the case of interest for applications to the CR Yamabe problem)
in dimension three,
but with the additional assumption that a certain fourth-order differential operator, the CR Paneitz operator $P_{\theta}$, is positive definite.
They showed that the pseudohermitian mass appears in the expansion of the Green function of the conformal sublaplacian in a way similar to Riemannian geometry, and as a result, they proved that if $P_{\theta}$ is positive definite, the functional $Q$ attains its minimum.

Later Takeuchi in \cite{T} proved that the positivity of $P_{\theta}$ is equivalent to embeddability in $\C^N$ for some $N$, a condition that always holds in dimension $n\ge 5$ by a theorem of Boutet de Monvel (proved in \cite{BdM}, see Theorem 12.2.1 in \cite{CS} for a more recent exposition) but may fail in dimension three.

In \cite{CMY2} Cheng, Malchiodi and Yang estimated the pseudohermitian mass
%the \rosso{blow-up}
of the Rossi spheres $S^3_s$, the simplest non-embeddable example of a three-dimensional CR manifold (see Subsection 12.4 in \cite{CS} for a proof of this fact) and showed that it satisfies $m_s=-18\pi s^2+o(s^2)$, and thus  is negative for small values of $s$.
Using this result they showed that on the Rossi spheres the CR Yamabe functional does not attain its minimum, and therefore that in order to solve in general the CR Yamabe problem, a non-minimizing proof like the one of Gamara and Yacoub is required.

These results are not the only ones that show a breakdown of the analogy between Riemannian and CR geometry which occurs only for non-embeddable CR manifolds (and therefore only in dimension three): in \cite{AP} the first two named authors built a CR structure on $S^3$ for which the set of solutions to the CR Yamabe problem is not compact (unlike Riemannian manifolds of dimension up to 24), and in \cite{ACMY} the first named author, Cheng, Malchiodi and Yang computed the variation (with respect to the CR structure) of the CR Einstein-Hilbert functional on $S^{2n+1}$, observing that for $n\ge 2$ its behavior is essentially the same as the analogous problem in Riemannian geometry, while on $S^3$ it behaves in the same way only on the subspace corresponding to embeddable CR structures, and in the opposite way on the orthogonal complement.

\vspace{5mm}

Motivated by the interest in the CR Yamabe problem, a flow analogous to the Yamabe flow has been introduced in order to study it: the so-called CR Yamabe flow is defined by
$$\frac{\de\theta(t)}{\de t} = -(R_{\theta(t)}-r_{\theta(t)})\theta(t).$$
or equivalently, writing $\theta(t)=u(t)^{\frac{2}{n}}\theta_0$,
$$\frac{\de u}{\de t} = \frac{n}{2}u^{-\frac{2}{n}}\left(b_n\Delta_{\theta_0}u-R_{\theta_0}u +r_{\theta(t)}u^{\frac{n+2}{n}} \right)$$
(see Section \ref{Section2} for the notation).
This flow has been shown to have a solution $\theta(t)$ for all $t\in[0,\infty)$ as proved by the third, together with Shen and Wang, in \cite{H-LTexistence}. Regarding convergence, in \cite{HSW} the third named author, Sheng and Wang proved a convergence theorem; since the positivity of the pseudohermitian mass was required, that theorem was formulated under the hypotheses of the available Positive CR Mass Theorems, that is the already cited one from \cite{CMY1} or \cite{CCY}. 
Note that other versions of Positive 
CR Mass Theorem have been proved in \cite{CC}
and \cite{H1}.
See also \cite{CC1,CCW,H3,H2,HW,SW}
and the references therein for results 
related to the CR Yamabe flow.

Since the Positive Mass Theorem fails for the pseudohermitian mass,
%(cf. \cite{CMY2}),
it is natural to ask whether convergence holds in general for the CR Yamabe flow. In this article we answer this question negatively.

\begin{teorema}\label{MainTheorem}
 There exists $s_0>0$ such that for every $s$ with $0<|s|<s_0$ there exists a smooth contact form $\theta_s$ for the Rossi sphere $S^3_s$ such that the CR Yamabe flow with initial condition $\theta_s$ does not converge.
\end{teorema}
To our knowledge, this provides the first natural non-convergence
example for the CR Yamabe flow. Our construction shows that negative
pseudohermitian mass is not merely an obstruction to the convergence of the flow; it gives rise to an instability of the flow
itself, forcing the solution toward a bubbling regime.

Let us describe the main idea behind the proof.
 The proof of convergence uses a concentration-compactness principle to prove that, up to subsequences, a non converging solution must blow up by forming bubbles at a certain number of points; namely, that there exist sequences $x_i^{\alpha}\in M$ and $\e_i^{\alpha}>0$ for $\alpha=1,\ldots,N$ such that
$$u(t_i) - \sum_{\alpha=1}^N\overline{u}_{x_i^{\alpha},\e_i^{\alpha}} \to u_{\infty}$$
in the Sobolev-type space $S^{1,2}(M)$, where $\overline{u}_{x_i^{\alpha},\e_i^{\alpha}}$ is a family of blowing up solutions of the problem on $\H^n$ called bubbles, transferred to $M$ through appropriate local coordinates, and then proving that the number of such bubbles is zero. Therefore, the simplest possible blow-up that could happen is one in which $u_{\infty}\equiv 0$ and there is only one bubble.
Hence we look for a counterexample in a neighborhood of  a very concentrated bubble.
By \cite{CMY2} we know that, in order to perform fine estimates for the CR Yamabe functional on bubbles on Rossi spheres, it is necessary to apply the Lyapunov-Schmidt method to the four-dimensional manifold $\ci{M}\subseteq S^{1,2}(S^3)$ of bubbles, getting a perturbed manifold $\widetilde{\ci{M}}=\{\Phi_{p,\lambda}^s\;|\;p\in S^3,\lambda>0\}$.
Inspired by computations in \cite{CMY2}, we prove the following estimate for the CR Yamabe functional $Q_s$
\begin{equation}\label{FormulaIntro}
    \frac{\de}{\de\lambda}Q_s(\Phi_{p,\lambda}^s) = \frac{4}{3\pi}\frac{m_s}{\lambda^3} + O\left(\frac{s^2}{\lambda^4}\right)
\end{equation}
 which, thanks to the negativity of $m_s$ and together with $U(2)$ invariance, suggests that the CR Yamabe flow for Rossi spheres starting in $\widetilde{\ci{M}}$ should blow up.
But $\widetilde{\ci{M}}$ is not invariant under the CR Yamabe flow, and so we define a neighborhood of it in $S^{1,2}(S^3)$
$$\Omega=\left\{\Phi_{p,\lambda}^s+v\;\middle|\; v\in V_{p,\lambda}, \N{v}<C\frac{s^2}{\lambda^2}, \lambda>\Lambda,p\in S^3
%,verde{\int(\Phi_{p,\lambda}+v)^4=\int\Phi_{p,\lambda}^4}
\right\}$$
where $V_{p,\lambda}=(T_{\phi_{p,\lambda}}\ci{M})^{\perp}$. Since $\ci{M}$ is the set of critical points of $Q_0$ and the CR Yamabe flow is essentially the $L^2$ gradient flow of $Q_s$, the invariance of $\Omega$ is governed by the second differential of $Q_0$ on $V_{p,\lambda}$.
$D^2Q_0(\phi_{p,\lambda})$ has been studied by Malchiodi and Uguzzoni in \cite{MU} using the spectral analysis of the sublaplacian on the CR sphere by Folland in \cite{F}; it is not positive definite on $V_{p,\lambda}$, but it is nondegenerate with Morse index one. We overcome the difficulty caused by the negative eigenspace by using the fact that the CR Yamabe flow preserves the volume in order to essentially restrict the flow to a codimension-one submanifold where $D^2Q_0(\phi_{p,\lambda})$ becomes positive definite.
Once proven that $\Omega$ is invariant by the CR Yamabe flow, since estimate \eqref{FormulaIntro} holds on $\Omega$, the fact that the flow with initial condition in $\Omega$ diverges follows easily.

We believe that this counterexample highlights the role of the
pseudohermitian mass in the dynamics of the CR Yamabe flow. In the
Riemannian setting, this role is harder to isolate, since the Positive
Mass Theorem rules out the negative-mass mechanism exploited here.

\vspace{15pt}
%\textbf{Plan of the paper:} The paper is organized as follows. Section $2$ collects the preliminaries on CR geometry and the CR Yamabe flow.  Section $3$ is devoted to the proof of the main theorem.

\textbf{Acknowledgements.}
P.T.H. was supported  by the National Science and Technology Council (NSTC),
Taiwan, with grant Number: 114-2115-M-032 -003 -MY2\\
A.P. is a member of the Istituto Nazionale di Alta Matematica (INdAM), Gruppo Nazionale per l'Analisi Matematica, la Probabilità e le loro Applicazioni (GNAMPA), and is supported by the University of Trento, the MIUR-PRIN 2022 Project \emph{Regularity problems in sub-Riemannian structures}  Project code: 2022F4F2LH and the INdAM-GNAMPA 2025 Project \emph{Structure of sub-Riemannian hypersurfaces in Heisenberg groups}, CUP ES324001950001.\\
C.A. is supported by the French National Research Agency (ANR) project  EINSTEIN-PPF, grants ANR-23-CE40-0010.

\section{Preliminaries in CR geometry}\label{Section2}
We recall some notions about CR manifolds. For a complete introduction we refer to the standard monograph \cite{DT}.
We recall that a CR structure on a $2n+1$-dimensional manifold $M$ is a complex $n$-dimensional subbundle $\ci{H}\subset TM\otimes\C$ such that $\ci{H}\cap\overline{\ci{H}}=\{0\}$ and $[\Gamma(\ci{H}),\Gamma(\ci{H})]\subset\Gamma(\ci{H})$.
Every real $2n+1$-dimensional real submanifold of $\C^{n+1}$ inherits the CR structure
$$(TM\otimes\C)\cap\operatorname{span}\left\{\frac{\de}{\de z^1},\ldots,\frac{\de}{\de z^{n+1}}\right\}.$$
A CR structure is called nondegenerate if $H(M)=\g{Re}(\ci{H}\oplus\overline{\ci{H}})$ is a contact distribution; in such a case a contact form is called a pseudohermitian structure. The choice of a contact form $\theta$ induces a linear connection called the Tanaka-Webster connection, and a nondegenerate symmetric bilinear product $G_{\theta}$ on $H(M)$; the CR manifold is called pseudoconvex when $G_{\theta}$ is positive definite. By contracting twice the restriction of curvature tensor of the Tanaka-Webster connection to $H(M)$ through $G_{\theta}$, a scalar curvature invariant $R$ called the Webster curvature arises.

Since a contact form for a given CR structure is unique up to multiplication by a nowhere zero function, it is natural to study conformal changes. If we write $\widetilde{\theta}=u^{\frac{2}{n}}\theta$, the Webster curvature associated to $\widetilde{\theta}$ is given by
\begin{equation}\label{ConformalChangeCurvature}
 \widetilde{R} = u^{-\frac{n+2}{n}}\left(-b_n\Delta_b +R \right)u
\end{equation}
where $b_n=2+\frac{2}{n}$ and $\Delta_bu= u_{,\alpha\overline{\alpha}}+u_{,\overline{\alpha}\alpha}$ is a second order subelliptic differential operator called the sublaplacian.
The operator $L_{\theta}=-b_n\Delta_b+R$ is called the conformal sublaplacian, and it satisfies the conformal covariance law
\begin{equation}\label{TransformationSublaplacian}
L_{\widetilde{\theta}}\left(\frac{\phi}{u}\right) = u^{-\frac{n+2}{n}}L_{\theta}\phi.
\end{equation}
Furthermore the choice of a contact form induces a volume form $dV_{\theta}=\theta\wedge(d\theta)^n$, satisfying the conformal transformation law
\begin{equation}\label{TransformationVolumeForm}
dV_{\widetilde{\theta}} = u^{2\frac{n+1}{n}}dV_{\theta}.
\end{equation}

The problem of finding contact forms with constant Webster curvature is equivalent to finding critical points of the CR Yamabe functional
\begin{equation}\label{CRYamabeFunctional}
 Q_{\theta}(u) = \frac{\int\widetilde{R}d\widetilde{V}}{\Vol_{\widetilde{\theta}}(M)}=\frac{\int\left(b_n|\nabla_bu|^2+Ru^2\right)dV}{\left(\int u^{\frac{2n+2}{n}} dV\right)^{\frac{n}{n+1}}} = \frac{\int uL_{\theta}udV}{\left(\int u^{\frac{2n+2}{n}} dV\right)^{\frac{n}{n+1}}}.
\end{equation}
The most important CR manifold is the Heisenberg group $\H^n$, defined as $\C^n\times\R$ with the group law $(z,t)\cdot(w,s) = (z+w,t+s+2\g{Im}(z\cdot\overline{w}))$, endowed with the left-invariant CR structure generated by the left-invariant vector fields
$$Z_{\alpha}= \frac{\de}{\de z_{\alpha}} +i\overline{z}_{\alpha}\frac{\de}{\de t}$$
and the contact form $\Theta= dt+i\sum_{\alpha}\left( z^{\alpha}d\overline{z}^{\alpha}-\overline{z}^{\alpha}dz^{\alpha}\right)$. $\H^n$ admits the family of CR and group isomorphisms $\delta_{\lambda}(z,t)=(\lambda z,\lambda^2t)$ for $\lambda>0$, called dilations.

The other fundamental CR manifold is $S^{2n+1}$ with the CR structure induced by $\C^{n+1}$, and carries the unique (up to a constant) $U(n)$-invariant contact form
$$\theta= \frac{i}{2}(\de-\overline{\de})|z|^2 = \frac{i}{2}\sum_{k=1}^{n+1}z^kd\overline{z}^k-\overline{z}^kdz^k$$
whose Webster curvature is a constant $\overline{R}_n$ by $U(n)$ invariance.

$\H^n$ and $S^{2n+1}$ are locally CR equivalent: in fact the map $F:S^{2n+1}\setminus\{(-1,0,\ldots,0)\}\to\H^n$ given by
$$F(z^1,\ldots,z^n)= \left(\frac{z^1}{1+z^{n+1}},\ldots,\frac{z^1}{1+z^{n+1}},\g{Re}\left(i\frac{1+z^{n+1}}{1-z^{n+1}}\right)\right),$$
called the Cayley map, is a CR equivalence.
The respective contact forms are related by
$$(F^{-1})^*\theta= U^{\frac{2}{n}}\Theta$$
where
$$U(z,t)=\frac{c_n}{\left(t^2+(1+|z|^2)^2\right)^{\frac{n}{2}}},$$
which by equation \eqref{ConformalChangeCurvature} is a solution of $-b_n\Delta_{\H^n}U= \overline{R}_nU^{\frac{n+2}{n-2}}$
Using the translations $\tau_xy=x\cdot y$ and the dilations we can get other solutions (called "bubbles") by $U_{x,\lambda}=\lambda^nU\circ\delta_{\lambda}\circ\tau_{x^{-1}}$.

We can use the Cayley transform to transfer the bubbles to $S^{2n+1}$: in the case $n=1$ from calling $U_{\lambda}=U_{0,\lambda}$ we get conformal factors $\phi_{\lambda}$ such that
$$\phi_{\lambda}^2\theta=F^*(U_{\lambda}^2\Theta)$$
and therefore solving
\begin{equation}
 L_{S^3}\phi_{\lambda}= 2\phi_{\lambda}^3.
\end{equation}
By composing $\phi_{\lambda}$ with elements of $SU(2)$ we get a finite dimensional submanifold $\ci{M}$ of $S^{1,2}$. Each element of $\ci{M}$ (except the constant $1$) is determined by the parameter $\lambda>1$ of the bubble $\phi_{\lambda}$ from which it is obtained and its maximum point $p$; we denote such a bubble by $\phi_{p,\lambda}$. Denote by $X_k$ for $k=1,2,3$ the vector fields associated to some basis of $\g{su}(2)$ with respect to the action of $SU(2)$ on $S^3$, the tangent space of $\ci{M}$ at $\phi_{p,\lambda}$ is
$$\operatorname{span}\left\{X_1\phi_{p,\lambda}, X_2\phi_{p,\lambda},X_3\phi_{p,\lambda},\frac{\de\phi_{p,\lambda}}{\de\lambda}\right\}.$$

\vspace{8mm}

Given a CR manifold with a contact form $\theta_0$, the CR Yamabe flow is the flow of contact forms $\theta(t)$ for $t\in[0,\infty)$ defined as
$$\left\{\begin{matrix}
    \frac{\de\theta}{\de t}(t) = -(R_{\theta(t)}-r_{\theta(t)})\theta(t)\\
    \theta(0)=\theta_0.
\end{matrix}\right.$$
Writing $\theta(t)=u(t)^{\frac{2}{n}}\theta_0$, this is equivalent to
\begin{equation}\label{CRYamabeFlowEquationGenrlDim}
 \frac{\de u}{\de t} = \frac{n}{2}\frac{1}{u^{\frac{2}{n}}}\left(-L_{\theta_0}u +r_{\theta(t)}u^{\frac{n+2}{n}} \right)
\end{equation}
where
$$r_{\theta(t)}=\frac{\int_MR_{\theta(t)}dV_{\theta(t)}}{\Vol_{\theta(t)}(M)} = \frac{\int_Mu(t)L_{\theta_0}u(t)dV_{\theta_0}}{\int_Mu(t)^{\frac{2n+2}{n}}dV_{\theta_0}}.$$
Since the linearization of the right hand side of \eqref{CRYamabeFlowEquationGenrlDim} is a second order subelliptic operator, it is standard to prove short-time existence and uniqueness using the heat kernel for $L_{\theta}$ and Duhamel's formula; see \cite{Z} for a proof performed by adding $\e\Delta_{g}$ (the Laplace-Beltrami operator with respect to a certain metric) to $L_{\theta}$, and by proving uniform estimates for $\e\to 0$. By \cite{H-LTexistence} the solution exists for all $t\in[0,\infty)$.

From now on we always assume $n=1$; in particular the CR Yamabe flow equation becomes
\begin{equation}\label{CRYamabeFlowEquation}
 \frac{\de u}{\de t} = \frac{1}{2}\frac{1}{u^2}\left(-L_{\theta_0}u +r_{\theta(t)}u^3 \right)
\end{equation}
Using this equation and formula \eqref{TransformationVolumeForm} it is not hard to verify that the volume of $M$ is invariant under the flow (for a proof see Proposition 3.1 in \cite{H-LTexistence}).

The CR structure of $S^3$ is generated by the vector field
$$Z= \overline{z_2}\frac{\de}{\de z_1} - \overline{z_1}\frac{\de}{\de z_2}.$$

The Rossi spheres $S^3_s$ are defined as $S^3$ with the CR structure generated by $Z+\frac{s}{\sqrt{1+s^2}}\overline{Z}$.
We call $L_s$ the conformal sublaplacian associated to this CR structure and the contact form $\theta$ of $S^3$ defined above.

We define the Folland-Stein space $S^{1,2}$ % \rosso{(OR ANOTHER NAME?)}
as the subspace of $W^{1,2}_{\rm{loc}}(S^3)$ for which the norm
$$\N{u}_{S^{1,2}}^2 = \int_MuL_0u$$
is finite. $S^{1,2}$ is a Hilbert space with the scalar product
$$\bra u,v\ket = \int_{S^3}uL_0v.$$
Notice that the norms $\N{u}_s^2 = \int uL_su$ are equivalent for $|s|\le M$ for any $M<1$, and we choose to use a fixed norm for convenience.

In \cite{JL} Jerison and Lee defined coordinates for a neighborhood of a point of a pseudoconvex CR manifold taking value in $\H^n$ and analogous to conformal normal coordinates in Riemannian geometry; these coordinates are called CR normal coordinates.

 Now suppose that the CR structure is of positive CR Yamabe class, that is, that $L_{\theta}$ is positive definite; then, given a point $p\in M$, $L_{\theta}$ has a Green function $G_p$ defined as the solution of $L_sG_p=64\pi\delta_p$.
 On $\H^1$ the Green function is given by $\frac{2}{\rho^2}$, where $\rho=(|z|^4+t^2)^{\frac{1}{4}}$ is called the Kor\'anyi norm; 
 it was proved in Proposition 5.2 in \cite{CMY1} that in general $G_p$ has the expansion in CR normal coordinates
 \begin{equation}\label{ExpansionGreenFunction}
  G_p=\frac{2}{\rho^2} + A + O(\rho)
 \end{equation}
 where $A$ is a constant.
 Similarly to Riemannian geometry, the constant $A$ is related to asymptotic geometry: the CR manifold $M\setminus\{p\}$ with the contact form $G_p^2\theta$ is asymptotically flat, in the sense that outside of a compact set it has coordinates with values in $H^1$ which approximate it well enough (see Definition 2.1 in \cite{CMY1} for a precise definition); those manifolds admit an invariant called pseudohermitian mass defines as
 $$m = i\lim_{R\to\infty}\int_{S_R}\omega^1_1\wedge\theta$$
 (see \cite{CMY1} for the notation), and the constant $A$ above is related to $m$ by
 \begin{equation}\label{MassGreenRelation}
     m=12\pi A.\footnote{In this paragraph there are differences in constants with \cite{CMY1} because we used the definition of the Green function as in \cite{CMY2}.}
 \end{equation}
In \cite{CMY2} the authors performed this computation for the Rossi spheres with small parameter $s$.

\begin{proposizione}\label{PropositionRossiMass}
 Let $m_s$ be the pseudohermitian mass of the manifold obtained from the Rossi sphere $S^3_s$ through the process defined above. Then, for $s\to 0$,
 $$m_s=-18\pi s^2+o(s^2).$$
\end{proposizione}
 
\section{Proof of Theorem \ref{MainTheorem}}
We recall Proposition 5.2 from \cite{CMY2}.

\begin{proposizione} 
%For $\phi_{\lambda}$ defined \rosso{ABOVE}
There exists a unique $w_\lambda \in S^{1,2}(S^3)$, depending 
smoothly on $\lambda$, such that $\|w_\lambda\|_{S^{1,2}(S^3)} \leq C \, s$ and satisfies 
\begin{equation*}
\int_{S^3} \phi_{\lambda}^2 \frac{\de\phi_{\lambda}}{\de\lambda} \, w_\lambda  \, 
   \theta\wedge d\theta = 0; \qquad \qquad  L_s (\phi_{\lambda} + w_\lambda)  - 2 
   (\phi_{\lambda} + w_\lambda)^3 
   = \ell \, \phi_\lambda^2 \frac{\de\phi_{\lambda}}{\de\lambda} 
\end{equation*}
for some $\ell \in \R$.
\end{proposizione}

Let
%$$V_{p,\lambda}=\operatorname{span}\left\{\frac{\de\phi_{a,\lambda}}{\de\lambda},\frac{\de\phi_{a,\lambda}}{\de a_i}\right\}^{\perp}$$
%\rosso{(define more clearly $\frac{\de\phi_{a,\lambda}}{\de a_i}$)},
$$V_{p,\lambda}=\operatorname{span}\left\{X_1\phi_{p,\lambda}, X_2\phi_{p,\lambda},X_3\phi_{p,\lambda},\frac{\de\phi_{p,\lambda}}{\de\lambda}\right\}^{\perp}$$
that is the orthogonal in $S^{1,2}$ of the tangent space of $\ci{M}$ at $\phi_{p,\lambda}$.

Let $\Phi_{p,\lambda}=\phi_{p,\lambda} + w_\lambda$. $\Phi_{p,\lambda}$ depends on $s$, as $w_\lambda$ does, but we will omit this dependence unless necessary.

Define
$$\Omega=\left\{\Phi_{p,\lambda}+v\;\middle|\; v\in V_{p,\lambda}, \N{v}<C\frac{s^2}{\lambda^2}, \lambda>\Lambda,p\in S^3
%,verde{\int(\Phi_{p,\lambda}+v)^4=\int\Phi_{p,\lambda}^4}
\right\}$$
for some $\Lambda$ (which will be chosen big enough in order that the following propositions hold).
%\verde{Notice that the last condition by \rosso{(PUT IN INTRO)} is invariant by the CR Yamabe flow.}

\begin{proposizione}\label{PropositionLateralBoundary}
 Let $u(t)$ be a solution of the CR Yamabe flow such that $u(t)\in\Omega$ for $t\in[0,t_0)$. Then, for $s$ sufficiently small
 $$u(t_0)\not\in\left\{\Phi_{p,\lambda}+v\;\middle|\; v\in V_{p,\lambda}, \N{v}=C\frac{s^2}{\lambda^2}, \lambda\ge\Lambda,p\in S^3 \right\}.$$
\end{proposizione}

\begin{proof}
 By hypothesis $u(t)=\Phi_{p(t),\lambda(t)}+v(t)$ with $\N{v(t)}<C\frac{s^2}{\lambda^2}$ and $\N{v(t_0)}=C\frac{s^2}{\lambda^2}$, which implies that
 \begin{align}
 \nonumber
 &0\le\left.\frac{d}{dt}\right|_{t=t_0}\N{v(t)}^2=2\bra v(t_0),\frac{\de v}{\de t}(t_0)\ket\\
 \nonumber
 &=2\bra v(t_0),\frac{\de u(t_0)}{\de t}-\frac{\de\Phi_{p(t_0),\lambda(t_0)}}{\de p}p'(t_0) -\frac{\de\Phi_{p(t_0),\lambda(t_0)}}{\de \lambda}\lambda'(t_0)\ket\\
\label{FormulaLemma1ter}
 &= 2\bra v(t_0),\frac{\de u(t_0)}{\de t}\ket + O\left(C\frac{s^4}{\lambda^4}\right).
 \end{align}
 But using the equation of the CR Yamabe flow \eqref{CRYamabeFlowEquation}, and setting \[\Phi=\Phi_{p(t_0),\lambda(t_0)},\ u=u(t_0)\ \mbox{and}\ v_0=v(t_0)\] we have 
 \begin{align}
 \nonumber
 &\bra v(t_0),\frac{\de u(t_0)}{\de t}\ket = \int \frac{\de u(t_0)}{\de t} L_sv(t_0)\\
 \nonumber
 &=\frac{1}{2}\int\frac{1}{u^2}\left(-L_{\theta}u +r_{t_0}u^3\right)L_sv(t_0)\\
 \nonumber
&=\frac{1}{2}\int\frac{1}{(\Phi+v_0)^2}\left(-L_s(\Phi+v_0) +r_{t_0}(\Phi+v_0)^3\right)L_sv_0\\
\label{FormulaLemma1quater}
 &=\frac{1}{2}\int\frac{1}{(\Phi+v_0)^2}\left(-L_sv_0 +r_{t_0}(\Phi+v_0)^3-2\Phi^3\right)L_sv_0
\end{align}
We have that
$$r_{t_0}=\frac{\int_MR_{\theta(t)}dV_{\theta(t)}}{\int_MdV_{\theta(t)}} = \frac{\int_MuL_sudV_0}{\int_Mu^4dV_0}=2+O\left(\frac{s^2}{\lambda^2}\right)$$
by \cite[Section B]{CMY2}. Then formula \eqref{FormulaLemma1quater} becomes
$$\bra v(t_0),\frac{\de u(t)}{\de t}\ket =$$
$$=\frac{1}{2}\int\frac{1}{(\Phi+v_0)^2}\left(-L_{\theta}v_0 +6(\Phi+v_0)^2v_0\right)L_{\theta}v_0 +O\left(C\frac{s^4}{\lambda^4}\right) .$$
Now let $\widetilde{\theta}=(\Phi+v_0)^2\theta$.
Then by formulas \eqref{TransformationSublaplacian} and \eqref{TransformationVolumeForm} $dV_{\widetilde{\theta}}=(\Phi+v_0)^4dV_{\theta}$ and
$$L_{\widetilde{\theta}}\left(\frac{v_0}{\Phi+v_0}\right)=(\Phi+v_0)^{-3}L_{\theta}v_0,$$
therefore we get
\begin{align}
\nonumber
&\bra v(t_0),\frac{\de u(t)}{\de t}\ket\\
\label{FormulaProofLemma1}
&=\frac{1}{2}\int\left(-L_{\widetilde{\theta}}\left(\frac{v_0}{\Phi+v_0}\right) +6\frac{v_0}{\Phi+v_0}\right)L_{\widetilde{\theta}}\left(\frac{v_0}{\Phi+v_0}\right)d\widetilde{V} +O\left(C\frac{s^4}{\lambda^4}\right).
\end{align}
%But $$\int\left(-L_{\widetilde{\theta}}\left(\frac{v_0}{\Phi+v_0}\right) +3r_0\frac{v_0}{\Phi+v_0}\right)L_{\widetilde{\theta}}\left(\frac{v_0}{\Phi+v_0}\right)$$
%$$\le -AC^2\N{\frac{v_0}{\Phi+v_0}}_{...} = -O\left(C^2\frac{s^4}{\lambda^4}\right)$$
%and therefore, by choosing $C$ large enough in the definition of $\Omega$, we get a contradiction 

Since the CR Yamabe flow preserves the total volume, we have
$$0 = \left.\frac{d}{dt}\right|_{t=t_0}\int u(t)^4 = 4\int u(t_0)^3\frac{\de u}{\de t}(t_0)=$$
$$=2\int u(t_0)\left(-L_su(t_0) + r_{t_0}u(t_0)^3\right).$$
Reasoning as above we get
$$0= \int(\Phi+v_0)\left(-L_sv_0 + 6\Phi^2v_0\right) +O\left(C^2\frac{s^4}{\lambda^4}\right) =$$
$$= \int\Phi\left(-L_sv_0 + 6\Phi^2v_0\right) +O\left(C^2\frac{s^4}{\lambda^4}\right) =$$
$$= -\int\Phi L_sv_0 + 6\int\Phi^3v_0 +O\left(C^2\frac{s^4}{\lambda^4}\right) =$$
$$= -\int\Phi L_sv_0 -3\int L_s\Phi v_0 +O\left(C^2\frac{s^4}{\lambda^4}\right) =$$
\begin{equation}\label{FormulaProofLemma1bis}
 =-4\int L_{\widetilde{\theta}}\frac{v(t)}{\Psi+v_0}d\widetilde{V} +O\left(C^2\frac{s^4}{\lambda^4}\right)=-4\bra 1,\frac{v(t)}{\Psi+v_0}\ket_{X} +O\left(C^2\frac{s^4}{\lambda^4}\right)
\end{equation}
where $X$ is the space where the norm
$$\N{u}_X^2 = \int uL_{\widetilde{\theta}}udV_{\widetilde{\theta}}$$
with the obvious Hilbert product.

It follows from the proof of Lemma 5 in \cite{MU}, despite not explicitly being stated there, that the operator $L_{\theta}-6\phi_{p,\lambda}^2$ is nondegenerate on $V_{p,\lambda}$ with Morse index one, and by conformal covariance and continuity also $L_{\widetilde{\theta}}-6$ is. Since $(L_{\widetilde{\theta}}-6)1 = -4$, formulas \eqref{FormulaProofLemma1} and \eqref{FormulaProofLemma1bis} imply that
$$\bra v(t_0),\frac{\de u(t)}{\de t}\ket< -\N{v}^2 + O\left(C\frac{s^4}{\lambda^4}\right) = -C^2\frac{s^4}{\lambda^4} + O\left(C\frac{s^4}{\lambda^4}\right);$$
but choosing $C$ large enough in the definition of $\Omega$ and also $\Lambda$ large enough, we get a contradiction with formula \eqref{FormulaLemma1ter}.
\end{proof}

\begin{proposizione}\label{GradientEstimate}
 In $\Omega$ the CR Yamabe functional (defined in \eqref{CRYamabeFunctional}) for the Rossi sphere $Q_s$ satisfies the estimate
 $$\frac{\de}{\de\lambda}Q_s(\Phi_\lambda +v) = \frac{4}{3\pi}\frac{m_s}{\lambda^3} + O\left(\frac{s^2}{\lambda^4}\right).$$
\end{proposizione}

\begin{proof}
To obtain this estimate, we will need the approximate critical points $\breve{\phi}_{\lambda}$ for $Q_s$ defined in Subsection B.1 in \cite{CMY2}.
 Given $p\in S^3$ and $r>0$ small, consider a function $F$ such that in CR normal coordinates
$$ \begin{cases}
  F(z,t) = |z|^2 & \hbox{ for } \rho\leq r; \\
  F \equiv 0 & \hbox{ for } \rho\geq 2 r
  \end{cases}$$
  where, we recall, $\rho^4=|z|^4+t^2$.
  Let $G_p$ be the Green function for $L_s$ centered at $p$ as defined in Section \ref{Section2}.
  
  Let $\widetilde{G}=G_p^{-2}$. Then we define
  $$\breve{\var}_{\lambda} = \frac{\l}{\left( 1 + \l^2 F + \l^4 \tilde{G} \right)^{\frac 12}}$$
  (we omit the dependence on $p$ and $s$ as superfluous).

  Let
\[
    S_\lambda=\int_{S^3}\breve\phi_\lambda L_s\breve\phi_\lambda\,
\theta\wedge d\theta,
\qquad
V_\lambda=\int_{S^3}\breve\phi_\lambda^4\,\theta\wedge d\theta .
\]
Then \(Q_s(\breve\phi_\lambda)=S_\lambda V_\lambda^{-1/2}\), and hence
\begin{equation}\label{FormulaDerivative}
    \frac{d}{d\lambda}Q_s(\breve\phi_\lambda)
=
S_\lambda' V_\lambda^{-1/2}
-\frac12 S_\lambda V_\lambda^{-3/2}V_\lambda' .
\end{equation}
In the proof of Lemma B.3 in \cite{CMY2} they prove that
$$S_{\lambda} = 2V_{\lambda} -\frac{3}{2}A\int_{\H^1} |z|^2 (4 + \lambda^2 |z|^2) \rho^2 \overset{\circ}{U}_\lambda^6\, \Theta\wedge d\Theta + $$
$$+\int_{S^3} \breve{\phi}^4 (O (\l^2 \rho^2) + O(\rho^3)) \th \wedge d \th + \int_{S^3} \breve{\phi}^6 (O(\rho^5) + O(\l^2 \rho^7))  \th \wedge d \th$$
where
\[
  \overset{\circ}{U}_\l = \frac{\l}{\left( 1 + \l^2 |z|^2 + \frac 14 \l^4 (|z|^4 + t^2) \right)^{\frac 12}}, 
  \qquad \quad (z,t) \in \H^1,
\]
and $A$ is defined in formula \eqref{ExpansionGreenFunction};
therefore, calling
$$S^1_{\lambda}= \int_{\H^1} |z|^2 (4 + \lambda^2 |z|^2) \rho^2 \overset{\circ}{U}_\lambda^6,$$
it holds that
$$S_{\lambda} = 2V_{\lambda} -\frac{3}{2}AS^1_{\lambda} +O\left(\frac{1}{\l^3}\right).$$
It can similarly be proved that
 \begin{align}
& \frac{dS_{\lambda}}{d\lambda} 
%\frac{\de}{\de\lambda}\int_{S^3} \breve{\phi}_{\lambda} L_s \breve{\phi}_\lambda \theta \wedge d \theta
=2\int_{S^3}\left(\frac{\de}{\de\lambda}\breve{\phi}_{\lambda} \right)L_s \breve{\phi}_\lambda \theta \wedge d \theta \nonumber\\
&=8 \int_{S^3} \breve{\phi}_\lambda^3\left(\frac{\de}{\de\lambda}\breve{\phi}_{\lambda} \right) \theta \wedge d \theta  - 9A \int_{\H^1} |z|^2 (4 + \lambda^2 |z|^2) \rho^2 \overset{\circ}{U}_\lambda^5\left(\frac{\de}{\de\lambda}\overset{\circ}{U}_{\lambda} \right) \, \Theta\wedge d\Theta \nonumber\\
&- 3\lambda A \int_{\H^1} |z|^4 \rho^2 \overset{\circ}{U}_\lambda^6 \, \Theta\wedge d\Theta + \int_{S^3} \breve{\phi}_{\lambda}^3\left(\frac{\de}{\de\lambda}\breve{\phi}_{\lambda} \right) (O (\lambda^2 \rho^2) + O(\rho^3)) \theta \wedge d \theta \nonumber\\
&+\int_{S^3} \breve{\phi}_{\lambda}^5\left(\frac{\de}{\de\lambda}\breve{\phi}_{\lambda} \right) (O(\rho^5) + O(\lambda^2 \rho^7))
  \theta \wedge d \theta . \nonumber
  \end{align}
hence
$$\frac{dS_{\lambda}}{d\lambda} = 2\frac{dV_{\lambda}}{d\lambda} - \frac{3}{2}A\frac{dS^1_{\lambda}}{d\lambda} + O\left(\frac{1}{\l^4}\right).$$
 In the proof of Lemma B.3 in \cite{CMY2} they prove that
$$\breve{\var}_\l = \overset{\circ}{U}_\l  + \frac 18 A \rho^6 \l^2 \overset{\circ}{U}_\l^3 
+ O \left( \frac{\rho^{12} \l^8}{(1 + \l^4 \rho^4)^2}  \right) \overset{\circ}{U}_\l$$
 and in a similar way it can be proved that
$$\frac{\de}{\de\lambda}\breve{\var}_\l = \frac{\de}{\de\lambda}\overset{\circ}{U}_\l  + \frac 14 A \rho^6 \l\overset{\circ}{U}_\l^3 + \frac{3}{8} A \rho^6 \l^2 \overset{\circ}{U}_\l^2\frac{\de}{\de\lambda}\overset{\circ}{U}_\l+$$
$$+ O \left( \frac{\rho^{12} \l^7}{(1 + \l^4 \rho^4)^2}  \right) \overset{\circ}{U}_\l + O \left( \frac{\rho^{12} \l^8}{(1 + \l^4 \rho^4)^2}  \right) \frac{\de}{\de\lambda}\overset{\circ}{U}_\l.$$
There they also prove that
\begin{eqnarray*}
  V_{\lambda} =\int_{S^3} \breve{\var}_\l^4 \th \wedge d \th  & = & \int_{\H^1} \overset{\circ}{U}_\l^4 \Theta\wedge d\Theta  + \frac{1}{2} A \l^2 \int_{\H^1} \rho^6 \overset{\circ}{U}_\l^6 \Theta\wedge d\Theta + O(1/\l^3)
\end{eqnarray*}
that is, calling
$$V_{\lambda}^1= \int_{\H^1} \rho^6 \overset{\circ}{U}_\l^6\Theta\wedge d\Theta ,$$
that
\begin{eqnarray*}
  V_{\lambda} & = & \int_{\H^1} \overset{\circ}{U}_\l^4 \Theta\wedge d\Theta  + \frac{1}{2} A \l^2 V_{\lambda}^1  + O(1/\l^3)
\end{eqnarray*}
and it can similarly be proved that
$$\frac{dV_{\lambda}}{d\lambda} = 4\int_{S^3} \breve{\var}_\l^3\frac{\de}{\de\lambda}\breve{\var}_\l \th \wedge d \th =$$
$$= A \l \int_{\H^1} \rho^6 \overset{\circ}{U}_\l^6 \Theta\wedge d\Theta + 
  3 A \l^2 \int_{\H^1} \rho^6 \overset{\circ}{U}_\l^5\frac{\de}{\de\lambda}\overset{\circ}{U}_\l \Theta\wedge d\Theta
  + O(1/\l^4)=$$
 $$= A\lambda V_{\lambda}^1 + \frac{1}{2}A\lambda^2\frac{dV_{\lambda}^1}{d\lambda} +O(1/\l^4).$$
They also prove that
$$\int_{\H^1} \overset{\circ}{U}_{\sqrt{2}}^4 \, \Theta\wedge d\Theta
= 4 \pi^2. $$
Putting all of the above estimates in formula \eqref{FormulaDerivative}, we get that
$$\frac{d}{d\lambda}Q_s(\breve\phi_\lambda)= V_{\lambda}^{-\frac{3}{2}}\left[\left( 2\frac{dV_{\lambda}}{d\lambda} - \frac{3}{2}A\frac{dS^1_{\lambda}}{d\lambda} \right)\left(4\pi^2 +\frac{1}{2} A \l^2 V_{\lambda}^1 \right) + \right.$$
$$\left.-\frac{1}{2}\left( 2V_{\lambda} -\frac{3}{2}AS^1_{\lambda}\right)\frac{dV_{\lambda}}{d\lambda}\right] +O\left(\frac{1}{\lambda^4}\right) =$$
$$= V_{\lambda}^{-\frac{3}{2}}\left[\left(2\frac{dV_{\lambda}}{d\lambda} - \frac{3}{2}A\frac{dS^1_{\lambda}}{d\lambda} \right) \cdot 4\pi^2-V_{\lambda}\frac{dV_{\lambda}}{d\lambda}\right] +O\left(\frac{1}{\lambda^4}\right) =$$
$$= (4\pi^2)^{-\frac{3}{2}}4\pi^2\left[A\lambda V_{\lambda}^1 + \frac{1}{2}A\lambda^2\frac{dV_{\lambda}^1}{d\lambda} - \frac{3}{2}A\frac{dS^1_{\lambda}}{d\lambda} \right] +O\left(\frac{1}{\lambda^4}\right) =$$
$$= \frac{1}{2\pi}A\left[\lambda\int_{\H^1} \rho^6 \overset{\circ}{U}_\l^6\Theta\wedge d\Theta + \frac{1}{2}\lambda^2\int_{\H^1}\rho^6 \cdot6\overset{\circ}{U}_\l^5\left(\frac{\de}{\de \lambda}\overset{\circ}{U}_\l\right)\Theta\wedge d\Theta +\right.$$
$$\left.-\frac{3}{2} \int_{\H^1} |z|^2\cdot 2\lambda\rho^2 \overset{\circ}{U}_\lambda^6\Theta\wedge d\Theta -\frac{3}{2} \int_{\H^1} |z|^2 (4 + \lambda^2 |z|^2) \rho^2 \cdot 6\overset{\circ}{U}_\lambda^5\left(\frac{\de}{\de \lambda}\overset{\circ}{U}_\l\right)\Theta\wedge d\Theta\right]=$$
$$= \frac{1}{2\pi}A\left( \l \int_{\H^1} \left(\rho^6 -3|z|^4 \rho^2\right)\overset{\circ}{U}_\l^6 \Theta\wedge d\Theta+ \right.$$
 $$\left.+  \int_{\H^1}\left(3\lambda^2\rho^6 -9|z|^2 (4 + \lambda^2 |z|^2) \rho^2\right) \overset{\circ}{U}_\lambda^5\left(\frac{\de}{\de\lambda}\overset{\circ}{U}_{\lambda} \right) \, \Theta\wedge d\Theta\right) +O\left(\frac{1}{\lambda^4}\right).$$

 But
 $$ \frac{d}{d\lambda}\int_{\H^1} \left( 3 |z|^2 (4 + \l^2 |z|^2) - \l^2 \rho^4 \right) \rho^2 \overset{\circ}{U}_\l^6 \, \Theta\wedge d\Theta = $$
 $$=\int_{\H^1} \left( 6\l |z|^4 - 2\l \rho^4 \right) \rho^2 \overset{\circ}{U}_\l^6 \, \Theta\wedge d\Theta + $$
 $$ +6\int_{\H^1} \left( 3 |z|^2 (4 + \l^2 |z|^2) - \l^2 \rho^4 \right) \rho^2 \overset{\circ}{U}_\l^5\frac{\de}{\de\lambda}\overset{\circ}{U}_{\lambda} \, \Theta\wedge d\Theta $$
 and in the proof of Lemma B.3 in \cite{CMY2} it is proved that
 $$\int_{\H^1} \left( 3 |z|^2 (4 + \l^2 |z|^2) - \l^2 \rho^4 \right) \rho^2 \overset{\circ}{\var}_\l^6 \, \overset{\circ}{\theta} \wedge d \overset{\circ}{\theta} = \frac{32\pi}{\lambda^2}.$$
 Therefore
 $$\frac{d}{d\lambda}Q_s(\breve{\var}_\l) = $$
 $$= A\frac{1}{2\pi}\left( \l \int_{\H^1} \left(\rho^6 -3|z|^4 \rho^2\right)\overset{\circ}{U}_\l^6 \Theta\wedge d\Theta + \right.$$
 $$\left.+ \int_{\H^1}\left(3\lambda^2\rho^6 -9|z|^2 (4 + \lambda^2 |z|^2) \rho^2\right) \overset{\circ}{U}_\lambda^5\left(\frac{\de}{\de\lambda}\overset{\circ}{U}_{\lambda} \right) \, \Theta\wedge d\Theta\right) +O\left(\frac{1}{\lambda^4}\right)=$$
 $$=-A\frac{1}{4\pi}\frac{d}{d\lambda}\int_{\H^1} \left( 3 |z|^2 (4 + \l^2 |z|^2) - \l^2 \rho^4 \right) \rho^2 \overset{\circ}{U}_\l^6 \, \Theta\wedge d\Theta +O\left(\frac{1}{\lambda^4}\right)=$$
 $$=-A\frac{1}{4\pi}\frac{d}{d\lambda}\left(\frac{32\pi}{\lambda^2}\right) +O\left(\frac{1}{\lambda^4}\right) = A\frac{16}{\lambda^3} +O\left(\frac{1}{\lambda^4}\right).$$
 The estimate is shown to hold for $\Phi_{\lambda}+v$ too by the same argument in Subsection B.3 in \cite{CMY2}. Then the thesis follows thanks to formula \eqref{MassGreenRelation}.
\end{proof}

\begin{lemma}\label{LemmaInvariant}
 Up to choosing $\Lambda$ big enough in the definition of $\Omega$, it is invariant by the CR Yamabe flow.
\end{lemma}

\begin{proof}
We have
$$dQ_s(u)[v]= \frac{2\int_MvL_su\int_Mu^4-\frac{1}{2}\int_MuL_su4\int_Mu^3v}{\left(\int_Mu^4\right)^{\frac{3}{2}}}=$$
\begin{equation}\label{DifferentialQ}
 =2\left(\int_Mu^4\right)^{-\frac{1}{2}}\int_M\left(L_su-r_{\theta}u^3\right)v
\end{equation}
By Proposition \ref{PropositionLateralBoundary} we have to prove that if $u(t)\in\Omega$ for $t\in[0,t_0)$ then
$$u(t_0)\not\in\left\{\Phi_{p,\Lambda}+v\;\middle|\; v\in V_{p,\Lambda}, \N{v}< C\frac{s^2}{\lambda^2},p\in S^3 \right\}.$$
Writing $u(t)=\Phi_{p(t),\lambda(t)}+v(t)$, if this were the case, projecting on the submanifold $\widetilde{\ci{M}}=\{\Phi_{p,\lambda}+v(t_0)\},$ then we would have
$$\left.\frac{d}{dt}\right|_{t=t_0}\Phi_{p(t),\lambda(t)}\in \left\{\g{v}\in
%V_{p(t_0),\Lambda}^{\perp}
T_{p(t_0),\Lambda}\widetilde{\ci{M}}
\;\middle|\;\bra\frac{\de\Phi_{p(t_0),\Lambda}}{\de\lambda},\g{v}\ket_{L^2}\le 0\right\}$$
Putting $u=\Phi_{p(t_0),\Lambda}$, and taking $v=\frac{\de\Phi_{p(t_0),\Lambda}}{\de\lambda}$ in formula \eqref{DifferentialQ} and $s$ is small enough we get a contradiction with Proposition \ref{GradientEstimate} and the negativity of the mass for Rossi spheres (Proposition \ref{PropositionRossiMass}).
\end{proof}

\begin{lemma}\label{LastLemma}
If $u(t)=\Phi_{p(t),\lambda(t)}+v(t)$ is a CR Yamabe flow in $\Omega$ then $\lambda(t)\gtrsim t^{\frac{1}{4}}$
\end{lemma}

\begin{proof}
 It follows from Proposition \ref{GradientEstimate} with an argument  similar to that of Lemma \ref{LemmaInvariant}.
\end{proof}

\begin{proof}[Proof of Theorem \ref{MainTheorem}]
 The theorem follows from Lemma \ref{LastLemma} and the density of smooth functions in $S^{1,2}$.
\end{proof}

\end{document}